\documentclass[12pt]{amsart}
\usepackage{amsmath, amssymb, amsthm, enumerate}
 \usepackage[displaymath, math lines, running]{lineno}

\makeatletter
\newtheorem*{rep@theorem}{\rep@title}\newcommand{\newreptheorem}[2]{%
\newenvironment{rep#1}[1]{%
\def\rep@title{\bf #2 \ref{##1}}%
\begin{rep@theorem}}%
{\end{rep@theorem}}}
\makeatother
\newreptheorem{theorem}{Theorem}
\newreptheorem{lemma}{Lemma}
\newreptheorem{corollary}{Corollary}

\usepackage[left=1.5in,top=1.25in,right=1.5in,bottom=1.25in]{geometry}

\numberwithin{equation}{section}
\newtheorem{theorem}{Theorem}[section]
\newtheorem{lemma}{Lemma}[section]

\newtheorem{corollary}[theorem]{Corollary}

\theoremstyle{definition}

\newtheorem{remark}{Remark}[section]

\newcommand{\qtq}[1]{\qquad\text{#1}\qquad}

\newcommand{\cyc}[1]{\langle #1\rangle}
\newcommand{\R}{\mathbb{R}}
\newcommand{\C}{\mathbb{C}}

\newcommand{\ep}{\epsilon}
\newcommand{\grad}{\nabla}
\renewcommand{\epsilon}{\varepsilon}

\title[Resolvent Estimates for 3D NLS with Inverse-Square Potential]{Resolvent Estimates for the 3D Schr\"odinger Operator with Inverse-Square Potential}
\author{A. Adam Azzam}
\date{\today}

\begin{document}
\bibliographystyle{plain}
\maketitle
\noindent
\begin{abstract}
We consider the unitary group $e^{-itH}$ for the Schr\"odinger operator with inverse-square potential $H:=-\Delta-\frac{1}{4|x|^2}$. We adapt Combes-Thomas estimates to show that, when restricted to non-radial functions, $H$ enjoys much better estimates that mirror those of the Laplacian. 

\end{abstract}
\tableofcontents

\newpage
\section{Introduction} In this paper, we will consider harmonic analysis questions concerning the unitary group $e^{-itH}$ for the Schr\"odinger operator with inverse-square potential\begin{align}\label{338}H:=-\Delta-\frac{1}{4|x|^2}\end{align} where $\Delta$ is the Laplacian in $\mathbb{R}^3$. 

The operator $H$ in \eqref{338} belongs to a family of operators 
\begin{align}\label{422}H_{a}:=-\Delta-\frac{a}{|x|^2}\qquad a\le \tfrac{1}{4} \end{align}
 of interest in quantum mechanics (see \cite{Case}, \cite{Kalf}). Since the potential term is homogenous of order $-2$, the potential and Laplacian scale exactly the same. For this reason, this makes the study of the associated PDEs 
 \begin{align}\label{ls}&\left\{\begin{array}{ll}(i\partial_t-H_{a})u=0,\qquad &(t,x)\in \R\times \R^3,\\ u(0,x)=u_0(x)\in \dot{H}_x^1(\mathbb{R}^3)\end{array}\right.\\ &\label{nls}\left\{\begin{array}{ll}(i\partial_t-H_{a})u=|u|^pu,\qquad &(t,x)\in \R\times \R^3,\\ u(0,x)=u_0(x)\in \dot{H}_x^1(\mathbb{R}^3)\end{array}\right.\end{align} of particular mathematical interest. Indeed, since the potential and Laplacian are of equal strength at every scale, analysis of \eqref{ls} is immune to standard perturbative methods. Thus, Schr\"odinger equations with inverse square potentials have necessitated a more delicate approach. This challenge has been met with a long campaign to understand those properties enjoyed by $H_{a}$ and the local and global behavior of solutions to $\eqref{nls}$. 
 
The choice of the constant $a=\tfrac{1}{4}$ in this paper is critical in another sense. This constant corresponds to the sharp constant in the Hardy inequality.
\begin{lemma}[Hardy Inequality]\label{hardy}If $u\in C_0^\infty(\mathbb{R}^3)$, then 
\begin{align}\label{blunthardyineq}\frac{1}{4}\int\frac{|u(x)|^2}{|x|^2}\ dx\le \int |\grad u(x)|^2\ dx. \end{align}Moreover, \eqref{blunthardyineq} fails if the constant $\frac{1}{4}$ is replaced by any constant $a>\frac{1}{4}$. \end{lemma} 
When $a\le \frac{1}{4}$ and $u\in C_0^\infty(\mathbb{R}^3)$ we see that \begin{align}\cyc{H_{a}u,u}_{L_x^2(\mathbb{R}^3)}=\int_{\mathbb{R}^3}|\grad u(x)|\ dx-a\int \tfrac{|u(x)|^2}{|x|^2}\ dx\ge 0, \end{align} and so $H_a$ is positive semi-definite. When $a>\tfrac{1}{4}$, all self-adjoint extensions of $H_{a}$ are unbounded below. In this case the spectrum of $H_{a}$ has infinitely many negative eigenvalues that diverge to $-\infty$ (\cite{Planchon}), making it impossible to obtain global dispersive estimates.

Substantial progress has been made in understanding $H_{a}$ and \eqref{nls} in both the critical case $a=1/4$ and the subcritical case $a<\frac{1}{4}$. The first wave of progress in the linear problem came in the form of studying the full range of dispersive and Strichartz estimates for solutions to \eqref{ls}. When $a\le \tfrac{1}{4}$, \cite{Burq}, Burq--Planchon--Stalker--Tahvildar-Zadeh demonstrated local smoothing for solutions to \eqref{ls}. The nonlinear problem in this critical regime was advanced in \cite{Suzuki}, when Suzuki demonstrated global wellposedness for \eqref{nls} in $\mathbb{R}^3$ when $a\le \frac{1}{4}$ and $p<5$ for a suitable class of initial data by proving Strichartz estimates for \eqref{ls} avoiding the sharp endpoint $(2,6)$. Recently, in \cite{Mizutani}, Mizutani recovers the sharp endpoint in a Lorentz space framework. 

In \cite{Kivi1}, Killip-Miao-Visan-Zhang-Zheng prove a variety of harmonic analysis tools for studying $H_a$ when $a\le \tfrac{1}{4}$. In particular, they develop a Littlewood-Paley theory for $H_a$, which allows them to prove in \cite{Kivi2} that \eqref{nls} with $p=5$ is globally well-posed when $a<\tfrac{1}{4}-\tfrac{1}{25}$. This represents the first progress made in treating the energy critical nonlinearity $|u|^4u$ in the presence of an inverse square potential.

\section{Main Results}

In this paper, we advocate for the idea that any problem involving the critical inverse square potential should be decomposed into a system of equations: one for the radial part of the solution (which is more susceptible to classical ODE techniques), and another for the remainder. Indeed, what is clear from previous advances in \cite{Burq}, \cite{Mizutani}, and \cite{Mizutani2} is that data in $L^2\ominus L_\text{rad}^2$ enjoy a wealth of useful estimates for $H_{a}$ which in turn yield Strichartz estimates that are not immediately available in the general case. When restricted to higher angular momentum, the operator $H_a$ enjoys estimates that mirror those of the Laplacian. 

To make this discussion more precise, let $P:L^2(\mathbb{R}^3)\to L_{\text{rad}}^2(\mathbb{R}^3)$ be the projection defined by \[Pf(x)=\frac{1}{4\pi^2}\int_{\mathbb{S}^2}f(|x|\omega)d\sigma(\omega).\]We define $P^{\perp}=I-P$ to be the orthogonal projection onto $L_{\perp}^2:=L^2\ominus L_{\text{rad}}^2$, the space of angular functions. As $P$ and $P^{\perp}$ commute with $\Delta$ and $-\frac{1}{4|x|^2}$, they commute with $H$. We will often write \[H^{\perp}:=P^{\perp}H=HP^{\perp}\] when convenient. 

The advantage of restricting to functions in $L_{\perp}^2$ is due, in part, to an improvement to Lemma \ref{hardy} in this setting. For a clear exposition, see \cite{Ekholm2006}.

\begin{lemma}[Improved Hardy Inequality]
\label{sharphardy} If $f:\mathbb{R}^3\to \C$ is a Schwartz function, then
\begin{align}
\label{sharphardyineq} \frac{9}{4}\int_{\mathbb{R}^3}\frac{|P^{\perp}f|^{2}}{|x|^2}\ dx\le \int_{\mathbb{R}^3}|\grad P^{\perp}f|^2\ dx.
\end{align}
Moreover, \eqref{sharphardyineq} fails if the constant $\frac{9}{4}$ is replaced by any constant $a>\frac{9}{4}$. \end{lemma}
For data in $L_\perp^2$ this improvement allows one to apply techniques available only in the subcritical regime $a<\frac{1}{4}$ at the critical constant $a=\frac{1}{4}$. In this paper, we exploit this improvement to expand upon the harmonic analysis tools developed in \cite{Kivi1}. We seek out which embeddings enjoyed by $\Delta$ are also enjoyed by $H$ in order to develop a Littlewood-Paley theory for $H$. Towards this end, we study the resolvent of $H$ and, in particular, how it differs from the resolvent of $\Delta$ when projected to angular functions.

In this paper we improve the resolvent estimate from \cite{Mizutani} to prove the following:
\begin{theorem}\label{1004}Let $(q,s)\in [1,\infty)\times [1,\infty)$ satisfy
\begin{align}
\tfrac{4}{3}\le \tfrac{1}{q}+\tfrac{1}{s'}\le \tfrac{5}{3}.
\end{align} If, additionally, \begin{align}\sqrt{2}>\max\{2-\tfrac{3}{s'}, 2-\tfrac{3}{q}, \tfrac{3}{q}-\tfrac{5}{2}, \tfrac{3}{s'}-\tfrac{5}{2}\},\end{align} then 
 \begin{align}\| P^{\perp}((H+1)^{-1}-(-\Delta+1)^{-1}) \|_{L_x^{q}(\mathbb{R}^3)\to L_{x}^{s}(\mathbb{R}^3)}<\infty.\end{align}
\end{theorem}
This estimate allows us to improve the range of Bernstein estimates enjoyed by $H^{\perp}$ that are found in \cite{Mizutani}. Indeed, we obtain that $H^{\perp}$ obeys many of the useful Bernstein estimates enjoyed by the Laplacian. More precisely, we prove the following:

\begin{corollary}[Sharp Sobolev Embedding] \label{134}If $f:\mathbb{R}^3\to \C$ is Schwartz, then \begin{align}\|P^{\perp}(H+1)^{-1}f \|_{L_x^6(\mathbb{R}^3)}\lesssim ||f||_{L_x^{2}(\mathbb{R}^3)}.\end{align}
\end{corollary}
\begin{corollary}\label{1005}If $f:\mathbb{R}^3\to \C$ is Schwartz, then \begin{align} \|P^{\perp}(H+1)^{-1}f \|_{L_x^\infty(\mathbb{R}^3)}&\lesssim ||f||_{L_x^{2}(\mathbb{R}^3)}\\  \|P^{\perp}(H+1)^{-1}f \|_{L_x^2(\mathbb{R}^3)}&\lesssim ||f||_{L_x^{1}(\mathbb{R}^3)}.\end{align}
\end{corollary}

\begin{theorem}[Bernstein Estimates] \label{1006}The operator $P^{\perp}e^{-H}:L^1(\mathbb{R}^3)\to L^\infty(\mathbb{R}^3)$ is bounded.
\end{theorem}

\begin{remark} The estimates in Theorem \ref{1004}, Corollary \ref{1005}, and Theorem \ref{1006} fail in the absence of the projection $P^{\perp}$. 

\end{remark}

%

\section{Acknowledgements}

It is difficult for me to articulate just how much I appreciate the guidance, support, and friendship of my advisors, Rowan Killip and Monica Visan. I am incredibly grateful to them for introducing me to this problem, for their time in discussing it with me, and for carefully reading this manuscript. This work was supported in part by NSF grant DMS 1265868 (P.I. Rowan Killip) and NSF grant DMS-1500707 (P.I. Monica Visan).

\section{Preliminaries}

We begin by fixing some notation. We will write $X\lesssim Y$ if there exists a constant $C$ so that $X\le CY$. When we wish to stress the dependence of this implicit constant on a parameter $\ep$ (say), so that $C=C(\ep)$, we write $X\lesssim_{\ep}Y$. We write $X\sim Y$ if $X\lesssim Y$ and $Y\lesssim X$. If there exists a {\em small} constant $c$ for which $X\le cY$ we will write $X\ll Y$. 

For $1\le r<\infty$ we recall the Lebesgue space $L_x^r(\R^3)$, which is the completion of smooth compactly supported functions $f:\R^3\to \C$ under the norm \[||f||_{r}=||f||_{L_x^r(\R^3)}:=\left({\int_{\R^3}|f(x)|^r\ dx}\right)^{\frac{1}{r}}.\] When $r=\infty$, we employ the essential supremum norm. 
Our convention for the Fourier transform on $\R^3$ is \[\hat{f}(\xi)=\frac{1}{(2\pi)^\frac{3}{2}}\int_{\R^3}e^{-ix\cdot \xi}f(x)\ dx.\] The Fourier transform allows us to define the fractional differentiation operators \[\widehat{|\grad|^sf}(\xi)=|\xi|^s\hat{f}(\xi)\qquad \text{ and }\qquad \widehat{\langle\grad \rangle^sf}(\xi)=(1+|\xi|^2)^\frac{s}{2}\hat{f}(\xi).\] The fractional differentiation operators give rise to the (in-)homogenous Sobolev spaces. We define $\dot{H}^{1,r}(\R^3)$ and $H^{1,r}(\R^3)$ to be the completion of smooth compactly supported functions $f:\R^3\to \C$ under the norms \[||f||_{\dot{H}^{1,r}(\R^3)}=|||\grad | f||_{L_x^r}\qquad \text{ and }\qquad ||f||_{{H}^{1,r}(\R^3)}=|||\cyc{\grad} f||_{L_x^r}.\] When $r=2$, we simply write $H^{1,r}=H^{1}$ and $\dot{H}^{1,r}=\dot{H}^1$. 

Let us define the operator $H=-\Delta-\tfrac{1}{4|x|^2}$ as the Friedrichs extension of the symmetric and non-negative sesquilinear form \begin{align}Q(u,v)=\int \grad u\cdot \overline{\grad v}-\frac{1}{4|x|^2}u\bar{v}\ dx\qtq{$u,v\in C_{0}^{\infty}(\mathbb{R}^3).$}\end{align} More explicitly, let $\bar{Q}$ be the closure of $Q$ with domain $D(\bar{Q})$ given by the completion of $C_0^\infty(\mathbb{R}^3)$ with respect to the norm $(\|u\|_{L_x^2(\mathbb{R}^3)}^2+Q(u,u))^\frac{1}{2}$. For each element $u$ in \[D(H)=\{u\in D(\bar{Q}): |\bar{Q}(u,v)|\le C_{u}\|v\|_{L_x^2(\mathbb{R}^3)}\text{ for all }v\in D(\bar{Q})\},\] we define $Hu$ to be the unique element in $L_x^2(\mathbb{R}^3)$ satisfying $\bar{Q}(u,v)=(Hu,v)$ for all $v\in D(\bar{Q})$.

\section{Combes-Thomas Estimates}
In this section we prove a resolvent estimate for $H$, of which a special case is an essential stratagem in \cite{Mizutani}. To do this, we adopt a technique due to Combes and Thomas (see \cite{CombesThomas}, \cite{Hislop}) for proving decay estimates for the resolvent. 

\begin{lemma}[Combes-Thomas Estimate, \cite{Hislop}] \label{238} Suppose $A$ is a self-adjoint operator on a Hilbert space $\mathcal{H}$ and its spectrum, $\sigma(A)$, is positive and satisfies $d:=\text{dist}(\sigma(A),0)>0$. If $B$ is self-adjoint on $\mathcal{H}$, then $A+i\beta B$ is self adjoint for all $\beta\in \R$ and \[\|(A+i\beta B)^{-1}\|\le d^{-1}.\]
\begin{proof} If $u\in \mathcal{H}$, then \begin{align*}
\|u\| \|(A+i\beta B)u\|\ge \text{Re}\cyc{u, (A+i\beta B)u}=\text{Re}(\cyc{u,Au}-i\beta \cyc{u,Bu})\ge d\|u\|^2. 
\end{align*}

\end{proof}
\end{lemma}

To use Lemma \ref{238} we will need a useful reformulation of Lemma \ref{sharphardy}.

\begin{corollary}If $f:\mathbb{R}^3\to \C$ is Schwartz and $P^{\perp}f=f$, then \begin{align}\label{819}\cyc{(H+\tfrac{8}{9}\Delta)f,f}\ge 0\end{align}and, likewise,  
\begin{align}\label{820}\cyc{(H-\tfrac{2}{|x|^2})f,f}\ge 0.\end{align}
\begin{proof} Suppose $f$ is Schwartz. As $P^{\perp}H=HP^{\perp}$, by \eqref{sharphardy} we see that \begin{align*}
\cyc{(H+\tfrac{8}{9}\Delta)f,f}_{L_x^2(\mathbb{R}^3)}&=\cyc{-\Delta f,f}_{L_x^2(\mathbb{R}^3)}-\cyc{\tfrac{1}{4|x|^2}f,f}_{L_x^2(\mathbb{R}^3)}+\cyc{\tfrac{8}{9}\Delta f,f}_{L_x^2(\mathbb{R}^3)}\\ &=\tfrac{1}{9}(\cyc{\grad f, \grad f}_{L_x^2(\mathbb{R}^3)}-\cyc{\tfrac{9}{4|x|^2}f,f}_{L_x^2(\mathbb{R}^3)})\\ &\ge 0.
\end{align*} Similarly, we see that 
\begin{align*}
\cyc{(H-\tfrac{2}{|x|^2})f,f}_{L_x^2(\mathbb{R}^3)}=\cyc{-\Delta f,f}_{L_x^2(\mathbb{R}^3)}-\cyc{\tfrac{9}{4|x|^2}f,f}_{L_x^2(\mathbb{R}^3)} &\ge 0.
\end{align*}
\end{proof}
\end{corollary}

\begin{theorem}[Combes-Thomas Estimates]\label{CombesThomas}For $p\in \R$, define \begin{align}
\label{857}A&:=1-p^2\left\{(H^{\perp}+1)^{-\tfrac{1}{2}}\tfrac{1}{|x|^2}(H^{\perp}+1)^{-\tfrac{1}{2}}\right\},\qtq{and}\\ 
\label{858}B&:=-ip(H^{\perp}+1)^{-\tfrac{1}{2}}\left\{\tfrac{x}{|x|^2}\cdot \grad+\grad \cdot \tfrac{x}{|x|^2}\right\}(H^{\perp}+1)^{-\tfrac{1}{2}}.
\end{align}If $|p|<\sqrt{2}$, then $A+iB$ is invertible. 

\begin{proof} First, note that $A$ and $B$ are self-adjoint and $\sigma(A)=\sigma(A)\cap \mathbb{R}^{+}$. By Lemma~\ref{238}, $A+iB$ is invertible provided that $d=\text{dist}(\sigma(A),0)>0$. To determine the range of $p$ for which $d>0$, it behooves us to compute
\[\|(H^{\perp}+1)^{-\tfrac{1}{2}}\tfrac{1}{|x|^2}(H^{\perp}+1)^{-\tfrac{1}{2}}\|_{L_x^2(\mathbb{R}^3)\to L_x^2(\mathbb{R}^3)}.\]
Towards this end we note that, by \eqref{820}, \[|x|(1+H^{\perp})^{\frac{1}{2}}(1+H^{\perp})^{\frac{1}{2}}|x|=|x|(1+H^{\perp})|x|\ge |x|(1+\tfrac{2}{|x|^2})|x|=2+|x|^2.\] It follows, then, that 
\begin{align}\label{1152}\|(H^{\perp}+1)^{-\frac{1}{2}}|x|^{-2}(H^{\perp}+1)^{-\frac{1}{2}}\|_{L_x^2(\mathbb{R}^3)\to L_x^2(\mathbb{R}^3)}\le \tfrac{1}{2}.\end{align} By a $TT^*$ argument, it follows that \begin{align}\label{905}
\||x|^{-1}(H^{\perp}+1)^{-\frac{1}{2}}\|_{L_x^2(\mathbb{R}^3)\to L_x^2(\mathbb{R}^3)}=\|(H^{\perp}+1)^{-\frac{1}{2}}|x|^{-1}\|_{L_x^2(\mathbb{R}^3)\to L_x^2(\mathbb{R}^3)}\le \tfrac{1}{\sqrt{2}}.
\end{align}Consequently, \begin{align}\label{1153}\|(H^{\perp}+1)^{-\frac{1}{2}}\tfrac{x}{|x|^2}\|_{L_x^2(\mathbb{R}^3)\to L_x^2(\mathbb{R}^3)}\le \tfrac{1}{\sqrt{2}}.\end{align}Thus, $d\ge 1-\frac{p^2}{2}$ so $d>0$ provided $|p|<\sqrt{2}$. 
%
%
It follows that $A+iB$ is invertible provided that $|p|< \sqrt{2}$, as desired.\end{proof}
\end{theorem}

\begin{corollary}[Weighted Resolvent Estimate]
\label{ResolventEst} The operator \[|x|^{1-p}(H^{\perp}+1)|x|^{1+p}:L_x^2(\mathbb{R}^3)\to L_x^2(\mathbb{R}^3)\] is bounded for all $p\in (-\sqrt{2},\sqrt{2})$. 

\end{corollary}
\begin{proof}
In the notation of Lemma \ref{CombesThomas}, with $A$ as defined in \eqref{857} and $B$ as defined in \eqref{858}, we see that 
\begin{align}
|x|^{1-p}(H^{\perp}+1)|x|^{1+p}&=|x|(H^{\perp}+1)^{\frac{1}{2}}\left\{A+iB\right\}(H^{\perp}+1)^{\frac{1}{2}}|x|.
\end{align}
By \eqref{905}, the operators $(H^{\perp}+1)^{\frac{1}{2}}|x|$ and $|x|(H^{\perp}+1)^{\frac{1}{2}}$ are boundedly invertible. By Lemma \ref{CombesThomas}, $A+iB$ is boundedly invertible provided that $p\in (-\sqrt{2},\sqrt{2})$. Thus, $|x|^{1-p}(H^{\perp}+1)|x|^{1+p}$ is bounded provided that $p\in (-\sqrt{2},\sqrt{2})$, as claimed. \end{proof}

\section{Kernel Estimates}
In this section we prove kernel estimates for the operator $P^{\perp}(-\Delta+1)^{-1}$. 

\begin{lemma}Fix $y\in \R^3$ and suppose $x\in \R^3$. If $2|x|<|y|$, then \begin{align}\label{MVTBound1}|P^{\perp}(|x|^{-1\pm p}(-\Delta+1)^{-1}\delta_y)|\lesssim |x|^{\pm p}|y|^{-2}e^{-|y|/4}.\end{align} If $2|y|<|x|$, then 
\begin{align}\label{MVTBound2}|P^{\perp}(|x|^{-1\pm p}(-\Delta+1)^{-1}\delta_y)|\lesssim |y||x|^{-3\pm p}e^{-|x|/4}.\end{align}
\begin{proof}As $|x|^{-1\pm p}$ is radial it commutes with $P^{\perp}$ and so \begin{align}\label{MVTWant}|P^{\perp}(|x|^{-1\pm p}(-\Delta+1)^{-1}\delta_y)|=|x|^{-1\pm p}\cdot |P^{\perp}((-\Delta+1)^{-1}\delta_y)|.\end{align}
It suffices, then, to bound \begin{align}P^{\perp}((-\Delta+1)^{-1}\delta_y)(x)=\int_{\mathbb{S}^2}\frac{e^{-|x-y|}}{|x-y|}-\frac{e^{-|\omega|x|-y|}}{|\omega |x|-y|}d\sigma(\omega),\end{align}where $d\sigma$ denotes the uniform probability measure on the unit sphere $\mathbb{S}^{2}$. By the mean value theorem, for each $\omega\in \mathbb{S}^{2}$ there exists some $c(\omega)$ between $|x-y|$ and $|\omega|x|-y|$ so that 
\begin{align}\left|\tfrac{e^{-|x-y|}}{|x-y|}-\tfrac{e^{-|\omega|x|-y|}}{|\omega |x|-y|}\right|&=\left|-e^{-c(\omega)}\left(\tfrac{c(\omega)+1}{c(\omega)^2}\right)(|x-y|-|\omega|x|-y|)\right|\\&\le e^{-c(\omega)}\left(\tfrac{c(\omega)+1}{c(\omega)^2}\right)\tfrac{2|x-\omega|x||\cdot |y|}{|x-y|+|\omega|x|-y|}.\label{MVTbound}
\end{align}
Suppose for now that $2|x|<|y|$. It follows, then, that 
\[c(\omega)\ge \min\{|x-y|,|\omega|x|-y|\}\ge ||x|-|y||>\tfrac{|y|}{2}.\]
Since the function $t\mapsto e^{-t}\left(\tfrac{t+1}{t^2}\right)$ is decreasing, we see that \begin{align}\label{cbound}e^{-c(\omega)}\left(\frac{c(\omega)+1}{c(\omega)^2}\right)\le e^{-|y|/2}\left(\tfrac{2|y|+4}{|y|^{2}}\right)\lesssim \tfrac{e^{-|y|/4}}{|y|^2}.\end{align} Moreover, we see that \begin{align}\label{CSbound}\tfrac{2|x-\omega|x||\cdot |y|}{|x-y|+|\omega|x|-y|}\le \tfrac{4|x||y|}{|y|}\le 4|x|.\end{align} Combining \eqref{MVTbound}, $\eqref{cbound}$, and $\eqref{CSbound}$ we get
\begin{align}\label{improvedMVT}\left|\tfrac{e^{-|x-y|}}{|x-y|}-\tfrac{e^{-|\omega|x|-y|}}{|\omega |x|-y|}\right|\lesssim |x||y|^{-2}e^{-|y|/4}.\end{align} Integrating over the sphere, we combine \eqref{improvedMVT} with \eqref{MVTWant} to deduce \eqref{MVTBound1}. 

If instead $2|y|<|x|$, then the left hand side of \eqref{improvedMVT} is bounded by the right hand side of \eqref{improvedMVT} with $|y|$ replaced with $|x|$ and \emph{vice versa}. This gives \eqref{MVTBound2}.
\end{proof}
\end{lemma}
\begin{lemma} \label{kernelbounds1} If $|p|<\tfrac{3}{2}$ and $y\in \R^3$, then  \[\|P^{\perp}(|x|^{-1\pm p}(-\Delta+1)^{-1}\delta_{y})\|_{L_x^2(\R^2)}\lesssim |y|^{\pm p}\min\{|y|^{-\frac{1}{2}},|y|^{-1}\}.\]
\begin{proof}Fix some $y\in \R^3$. We split the integral into three regions \begin{align}
\left(\int_{A}+\int_{B}+\int_{C}\right)|P^{\perp}(|x|^{-1\pm p}(-\Delta+1)^{-1}\delta_{y})|^{2}\ dx=I_1+I_2+I_3,
\end{align} where $A=\{2|x|<|y|\}$, $B=\{|x|>2|y|\}$, and $C=\{|y|/2\le |x|\le 2|y|\}$. On the region $A$, we employ $\eqref{MVTBound1}$ to see that \begin{align*}
\int_{A}|P^{\perp}(|x|^{-1\pm p}(-\Delta+1)^{-1}\delta_{y})|^{2}\ dx&\lesssim |y|^{-4}e^{-|y|/2}\int_{A}|x|^{\pm 2p}\ dx\\ &\lesssim  e^{-|y|/2}|y|^{-1\pm 2p}\\ &\lesssim |y|^{\pm 2p}(|y|^{-2}\wedge |y|^{-1}),
\end{align*} provided that $|p|<\frac{3}{2}$. 

On the region $B$, we employ $\eqref{MVTBound2}$ to see that \begin{align*}
\int_{B}|P^{\perp}(|x|^{-1\pm p}(-\Delta+1)^{-1}\delta_{y})|^{2}\ dx&\lesssim |y|^{2}\int_{B}|x|^{-6\pm 2p}e^{-|x|/2}\ dx\\ &\lesssim |y|^{2}e^{-|y|}\int_{2|y|}^{\infty}r^{-4\pm 2p}\ dr\\ &\lesssim |y|^{-1\pm 2p}e^{-|y|}\\ &\lesssim |y|^{\pm 2p}(|y|^{-2}\wedge |y|^{-1}),
\end{align*}provided that $|p|<\tfrac{3}{2}$. 

Since $P^{\perp}:L_x^2(C)\to L_x^2(C)$ is bounded, we see that 
\begin{align}
\int_{C}|P^{\perp}(|x|^{-1\pm p}(-\Delta+1)^{-1}\delta_{y})|^{2}\ dx&\lesssim \int_{C}|x|^{-2\pm 2p}|(-\Delta+1)^{-1}\delta_{y})|^{2}\ dx\\ &\lesssim \int_{C}|x|^{-2\pm 2p}\tfrac{e^{-2|x-y|}}{|x-y|^2} dx\\ &\lesssim |y|^{-2\pm 2p}\int_{C}\tfrac{e^{-2|x-y|}}{|x-y|^2} \ dx.
\end{align} We may further split $C=C_1\cup C_2$ where $C_1=\{|x-y|<\tfrac{|y|}{2}\}$ and $C_2=C\setminus C_1$. On $C_1$ we see by a change of coordinates that \begin{align}\int_{C_1}\frac{e^{-2|x-y|}}{|x-y|^2}\ dx\lesssim \int_{0}^{\frac{|y|}{2}}e^{-2r}\ dr=1-e^{-|y|}\lesssim (1 \wedge |y|).\end{align} On $C_2$ we know that $|x-y|\ge \frac{|y|}{2}$ and so \begin{align}\int_{C_2}\frac{e^{-2|x-y|}}{|x-y|^2}\ dx \lesssim \frac{e^{-|y|}}{|y|^2}\int_{C_2}1 \ dx\lesssim |y|e^{-|y|}\lesssim (1\wedge |y|).\end{align} It follows, then, that \[\int_{C}|P^{\perp}(|x|^{-1\pm p}(-\Delta+1)^{-1}\delta_{y})|^{2}\ dx\lesssim |y|^{\pm 2p}(|y|^{-2}\wedge |y|^{-1}).\] Summing over the regions, we have our desired bound. \end{proof}
\end{lemma}

\begin{corollary}\label{1038}The integral kernel of $P^{\perp}((H+1)^{-1}-(-\Delta+1)^{-1})$, namely,\begin{align}G(z,y)=\cyc{\delta_{z},P^{\perp}((H+1)^{-1}-(-\Delta+1)^{-1})\delta_y}, \end{align} obeys the bound \begin{align}|G(z,y)|\lesssim \left(\tfrac{|z|\wedge |y|}{|z|+|y|}\right)^p(|z|^{-1}\wedge |z|^{-\frac{1}{2}})(|y|^{-1}\wedge |y|^{-\frac{1}{2}})\end{align} for all $p\in (-\sqrt{2},\sqrt{2})$. 

\begin{proof} Given two invertible operators $A,B$, we have the resolvent identity \begin{align}\label{resolvent}A^{-1}-B^{-1}=B^{-1}(B-A)B^{-1}+B^{-1}(B-A)A^{-1}(B-A)B^{-1}.\end{align} As $(-\Delta+1)-(H+1)=\tfrac{1}{4}|x|^{-2}$, by \eqref{resolvent} we see that 
\begin{align}P^{\perp}(H+1)^{-1}-P^{\perp}(-\Delta+1)^{-1}&=\text{(I)} + \text{(II)},\end{align}
where 
\begin{align}\text{(I)}&:= \tfrac{1}{4}\cdot P^{\perp}(-\Delta+1)^{-1}|x|^{-2}P^{\perp}(-\Delta+1)^{-1}\qtq{and}\\ \text{(II)}&:=\tfrac{1}{16}\cdot P^{\perp}(-\Delta+1)^{-1}|x|^{-2}P^{\perp}(H+1)^{-1}|x|^{-2}P^{\perp}(-\Delta+1)^{-1}.\end{align}Let $G_{(\text{I})}(z,y)$ and  $G_{(\text{II})}(z,y)$ be the integral kernels of (I) and (II), respectively, defined as above. It follows, then, from Cauchy-Schwarz and \eqref{kernelbounds1} that \begin{align}|G_{\text{(I)}}(z,y)|&\lesssim \langle |x|^{-1-p}P^{\perp}(-\Delta+1)^{-1}\delta_{z}, |x|^{-1+p}P^{\perp}(-\Delta+1)^{-1}\delta_{y}\rangle_{L_x^2(\mathbb{R}^2)}\nonumber\\ &\lesssim \||x|^{-1-p}P^{\perp}(-\Delta+1)^{-1}\delta_{z}\|_{L_x^2(\mathbb{R}^3)}\||x|^{-1+p}P^{\perp}(-\Delta+1)^{-1}\delta_{y}\|_{L_x^2(\mathbb{R}^3)}\nonumber\\ &\lesssim |z|^{-p} |y|^{p}(|y|^{-\frac{1}{2}}\wedge |y|^{-1})(|z|^{-\frac{1}{2}}\wedge |z|^{-1}).
\end{align}
Since $(I)$ is self-adjoint, $G_{\text{(I)}}$ is conjugate symmetric and hence obeys the same bounds with the roles of $|z|$ and $|y|$ reversed. So \begin{align}
\label{kernelboundsg1}|G_{\text{(I)}}(z,y)||&\lesssim \left(\tfrac{|z|\wedge |y|}{|z|+|y|}\right)^p(|z|^{-1}\wedge |z|^{-\frac{1}{2}})(|y|^{-1}\wedge |y|^{-\frac{1}{2}}),
\end{align}provided that $|p|<\tfrac{3}{2}$. For $|p|<\sqrt{2}$, we know by Corollary \ref{ResolventEst} that \begin{align*}
C_{p}:=\||x|^{-1+p}P^{\perp}(H+1)^{-1}|x|^{-1-p}\|_{L_x^2(\R^3)\to L_x^2(\R^3)}\lesssim_{p} 1.
\end{align*} It follows then by $\eqref{kernelbounds1}$, as before, that \begin{align}
|G_{\text{(II)}}(z,y)|&\lesssim C_{p}\||x|^{-1-p}P^{\perp}(-\Delta+1)^{-1}\delta_{z}\|_{L_x^2(\mathbb{R}^3)}\||x|^{-1+p}P^{\perp}(-\Delta+1)^{-1}\delta_{y}\|_{L_x^2(\mathbb{R}^3)}\nonumber\\ &\lesssim
C_{p}|z|^{-p} |y|^{p}(|y|^{-\frac{1}{2}}\wedge |y|^{-1})(|z|^{-\frac{1}{2}}\wedge |z|^{-1}).
\end{align} Again by conjugate symmetry of $G_{\text{(II)}}$, we see that $G_{\text{(II)}}$ obeys $\eqref{kernelboundsg1}$, as claimed. 
\end{proof}
\end{corollary}

\section{Proofs of Main Results}
In this section we begin by proving Theorem \ref{1004}, and conclude by deducing a number of important Bernstein estimates. 

\begin{reptheorem}{1004}
\label{1150} Let $(q,s)\in [1,\infty)\times [1,\infty)$ satisfy
\begin{align}
\label{1029}\tfrac{4}{3}\le \tfrac{1}{q}+\tfrac{1}{s'}\le \tfrac{5}{3}.
\end{align} If, additionally, \begin{align}\label{1030}\sqrt{2}>\max\{2-\tfrac{3}{s'}, 2-\tfrac{3}{q}, \tfrac{3}{q}-\tfrac{5}{2}, \tfrac{3}{s'}-\tfrac{5}{2}\},\end{align} then 

 \begin{align}\label{1028}\| P^{\perp}((H+1)^{-1}-(-\Delta+1)^{-1}) \|_{L_x^{q}(\mathbb{R}^3)\to L_{x}^{s}(\mathbb{R}^3)}<\infty.\end{align}

\begin{proof} We will only treat the case $q,s\in [1,2]\times [1,2]$. For if $(q,s)\in [2,\infty)\times [2,\infty)$ satisfy \eqref{1029} and \eqref{1030}, then $(s',q')$ belongs to $[1,2]\times [1,2]$ and satisfy $\eqref{1029}$ and $\eqref{1030}$. If the theorem is known to be true in this case, then \[\| P^{\perp}((H+1)^{-1}-(-\Delta+1)^{-1}) \|_{L_x^{s'}(\mathbb{R}^3)\to L_{x}^{q'}(\mathbb{R}^3)}<\infty,\] but then by duality it must also follow that \[\| P^{\perp}((H+1)^{-1}-(-\Delta+1)^{-1}) \|_{L_x^{q}(\mathbb{R}^3)\to L_{x}^{s}(\mathbb{R}^3)}<\infty.\] 

Now, by duality, it suffices to show that \[\sup_{||g||_{L_y^{s'}(\mathbb{R}^3)}\le 1}\sup_{ ||f||_{L_x^{q}(\mathbb{R}^3)}\le 1}\left|\iint_{\mathbb{R}^3\times \mathbb{R}^3}G(x,y)f(x)g(y)\ dx dy\right|<\infty.\] Suppose, then, that $||f||_{L_x^q(\mathbb{R}^3)}\le 1$ and $||g||_{L_y^{s'}(\mathbb{R}^3)}\le 1$. As $G$ depends only on $|x|$ and $|y|$ we may, and do, assume without loss of generality that $g=g(|y|)$ and $f=f(|x|)$ are radial. There are precisely six regions to consider, each case corresponding to an ordering of $|x|$, $|y|$, and $1$ (e.g. $|x|\ge |y|\ge 1$, $|x|\ge 1\ge |y|$, etc.). By symmetry, it suffices to treat the three cases in which we always have $|x|\ge |y|$.  
\begin{enumerate}[(1)]

\item With $A=\{(x,y)\in \R^3\times \R^3: |x|\ge |y|\ge 1\}$, we have by Corollary \ref{1038}, for any $p\in (-\sqrt{2},\sqrt{2})$, that \begin{align}&\int_{A}|G(x,y)f(x)g(y)|\ dx dy\nonumber \\ &\lesssim \int_{1}^{\infty} \int_{1}^{r} r^{1-p}\rho^{1+p}|f(r)||g(\rho)|d\rho dr\nonumber \\ &\lesssim \int_{1}^{\infty} \int_{1}^{r} r^{1-\frac{2}{q}-p}\rho^{1-\frac{2}{s'}+p}|f(r)r^{\frac{2}{q}}||g(\rho)\rho^{\frac{2}{s'}}|d\rho dr\nonumber \\ &\lesssim \sum_{k=0}^{\infty}\sum_{j=0}^{k}2^{k(1-\frac{2}{q}+\frac{1}{q'}-p)}2^{j(1-\frac{2}{s'}+\frac{1}{s}+p)}||f||_{L_x^{q}(|x|\sim 2^{k})}||g||_{L_y^{s'}(|y|\sim 2^{j})}\nonumber \\ &= \sum_{0\le j\le k}\left(\tfrac{2^j}{2^k}\right)^{p+\frac{3}{q}-2}2^{j(4-\frac{3}{q}-\frac{3}{s'})}||f||_{L_x^{q}(|x|\sim 2^{k})}||g||_{L_y^{s'}(|y|\sim 2^{j})}
\end{align} 
We may choose $p$ so that $2-\frac{3}{q}-p<0$. Since $\frac{1}{q}+\frac{1}{s'}\ge \frac{4}{3}$ we see that by Schur's Test  \begin{align}\int_{A}|G(x,y)f(x)g(y)|\ dx dy&\le \left(\sum_{k=0}^{\infty}||f||_{L_x^{q}(|x|\sim 2^{k})}^2\right)^{\frac{1}{2}}\left(\sum_{k=0}^{\infty}||g||_{L_y^{s'}(|x|\sim 2^{j})}^2\right)^{\frac{1}{2}}\nonumber \\ &\lesssim ||f||_{L_x^{q}(\mathbb{R}^3)}||g||_{L_y^{s'}(\mathbb{R}^3)}\nonumber \\ &\lesssim 1,\end{align}
since $q, s'\le 2$. 

\item With $B=\{(x,y)\in \R^3\times \R^3: |x|\ge 1\ge |y|\}$, we similarly have by Corollary \ref{1038}, for any $p\in (-\sqrt{2},\sqrt{2})$, that 
\begin{align}&\int_{B}|G(x,y)f(x)g(y)|\ dx dy\nonumber \\ &\lesssim \int_{1}^{\infty} \int_{0}^{1} r^{1-\frac{2}{q}-p}\rho^{\frac{3}{2}-\frac{2}{s'}+p}|f(r)r^{\frac{2}{q}}||g(\rho)\rho^{\frac{2}{s'}}|d\rho dr\nonumber \\ &\lesssim \sum_{k=0}^{\infty}\sum_{j=0}^{\infty}2^{k(1-\frac{2}{q}+\frac{1}{q'}-p)}2^{-j(\frac{3}{2}-\frac{2}{s'}+\frac{1}{s}+p)}||f||_{L_x^{q}(|x|\sim 2^{k})}||g||_{L_y^{s'}(|y|\sim 2^{j})}.\end{align} Choosing $p$ so that both $\frac{5}{2}-\frac{3}{s'}+p>0$ and, simultaneously, $2-\frac{3}{q}-p<0$, by Cauchy-Schwarz, we see that \begin{align}
\int_{B}|G(x,y)f(x)g(y)|\ dx dy\lesssim ||f||_{L_x^{q}(\mathbb{R}^3)}||g||_{L_y^{s'}(\mathbb{R}^3)}\lesssim 1.
\end{align}

\item With $C=\{(x,y)\in \R^3\times \R^3: 1\ge |x|\ge |y|\}$, we similarly compute that  \begin{align}&\int_{C}|G(x,y)f(x)g(y)|\ dx dy\nonumber \\ &\lesssim \int_{0}^{1} \int_{0}^{r} r^{\frac{3}{2}-p}\rho^{\frac{3}{2}+p}|f(r)||g(\rho)|d\rho dr\nonumber \\ &\lesssim \int_{0}^{1} \int_{0}^{r} r^{\frac{3}{2}-\frac{2}{q}-p}\rho^{\frac{3}{2}-\frac{2}{s'}+p}|f(r)r^{\frac{2}{q}}||g(\rho)\rho^{\frac{2}{s'}}|d\rho dr\nonumber \\ &\lesssim \sum_{0\le k\le j}\left(\tfrac{2^{-j}}{2^{-k}}\right)^{\tfrac{5}{2}-\tfrac{3}{s'}+p}2^{-k(5-\tfrac{3}{s'}-\tfrac{3}{q})}||f||_{L_x^{q}(|x|\sim 2^{k})}||g||_{L_y^{s'}(|y|\sim 2^{j})}\end{align}

We may choose a suitable $p$ satisfying $\tfrac{5}{2}-\tfrac{3}{s'}+p>0$. As $\tfrac{3}{s'}+\tfrac{3}{q}\le 5$, by Schur's Test we get 
\begin{align}\int_{C}|G(x,y)f(x)g(y)|\ dx dy\lesssim 1,
  \end{align}as desired.
\end{enumerate}

\end{proof}
\end{reptheorem}

We now pause to enjoy a few specific dividends of Theorem \ref{1004}. 

\begin{repcorollary}{134}[Sharp Sobolev Embedding] If $f:\mathbb{R}^3\to \C$ is Schwartz, then \begin{align}\|P^{\perp}(H+1)^{-1}f \|_{L_x^6(\mathbb{R}^3)}\lesssim ||f||_{L_x^{2}(\mathbb{R}^3)}.\end{align}
\end{repcorollary}
\begin{proof}  If $f:\mathbb{R}^3\to \C$ is Schwartz, then 
\begin{align}&\|P^{\perp}(H+1)^{-1}f \|_{L_x^6(\mathbb{R}^3)}\\ &\le \|P^{\perp}(-\Delta+1)^{-1}f \|_{L_x^6(\mathbb{R}^3)}+\|P^{\perp}((H+1)^{-1}-(-\Delta+1)^{-1})f \|_{L_x^{6}(\mathbb{R}^3)}
\\ &\lesssim ||f||_{L_x^{2}(\mathbb{R}^3)}\end{align}
 by Theorem \ref{1004} since the pair $(2,6)$ satisfies \[\tfrac{4}{3}=\tfrac{1}{2}+\tfrac{5}{6}\le \tfrac{5}{3}\qtq{and} \sqrt{2}>\max\{-\tfrac{1}{2},\tfrac{1}{2},-1,0\}.\]
\end{proof}

\begin{repcorollary}{1005}If $f:\mathbb{R}^3\to \C$ is Schwartz, then \begin{align}\label{1203} \|P^{\perp}(H+1)^{-1}f \|_{L_x^\infty(\mathbb{R}^3)}&\lesssim ||f||_{L_x^{2}(\mathbb{R}^3)}\\ \label{1204} \|P^{\perp}(H+1)^{-1}f \|_{L_x^2(\mathbb{R}^3)}&\lesssim ||f||_{L_x^{1}(\mathbb{R}^3)}.\end{align}
\end{repcorollary}
\begin{proof}  To see the first inequality, note that if $f:\mathbb{R}^3\to \C$ is Schwartz, then 
\begin{align*}&\|P^{\perp}(H+1)^{-1}f \|_{L_x^\infty(\mathbb{R}^3)}\\ &\le \|P^{\perp}(-\Delta+1)^{-1}f \|_{L_x^\infty(\mathbb{R}^3)}+\|P^{\perp}((H+1)^{-1}-(-\Delta+1)^{-1})f \|_{L_x^{\infty}(\mathbb{R}^3)}
\\ &\lesssim ||f||_{L_x^{2}(\mathbb{R}^3)}.\end{align*}
 by Theorem \ref{1004} since the pair $(2,\infty)$ satisfies \[\tfrac{4}{3}\le \tfrac{1}{2}+1\le \tfrac{5}{3}\qtq{and} \sqrt{2}>\max\{-1,\tfrac{1}{2},-1,\tfrac{1}{2}\}.\] The second inequality follows from duality. 
\end{proof}

\begin{reptheorem}{1006}[Bernstein Estimates] The operator $P^{\perp}e^{-H}:L^1\to L^\infty$ is bounded.
\begin{proof}
First, write \[P^{\perp}e^{-H}=(H+1)^{-1}P^{\perp}e^{-H}(H+1)^2(H+1)^{-1}P^{\perp}.\] 
It follows that \begin{align*}\|P^{\perp}e^{-H}\|_{L_x^1\to L_x^\infty}&=\|(H+1)^{-1}P^{\perp}e^{-H}(H+1)^2(H+1)^{-1}P^{\perp}\|_{L_x^1\to L_x^\infty}\\ &\le \|(H+1)^{-1}P^{\perp}\|_{L_x^2\to L_x^\infty} \|e^{-H}(H+1)^2\|_{L_x^2\to L_x^2} \|(H+1)^{-1}P^{\perp}\|_{L_x^1\to L_x^2}. \end{align*} By \eqref{1203}, $\|(H+1)^{-1}P^{\perp}\|_{L_x^2\to L_x^1}<\infty$. By \eqref{1204}, $ \|(H+1)^{-1}P^{\perp}\|_{L_x^1\to L_x^2}<\infty$. Lastly, $\|e^{-H}(H+1)^2\|_{L_x^2\to L_x^2}<\infty$ by the Spectral Theorem.
\end{proof}
\end{reptheorem}
%
%
%
\newpage

\bibliography{ResolventEstimates}
\end{document}